\documentclass[a4paper]{article}

\usepackage{amssymb,latexsym,amsmath,mathrsfs,epsfig,amsthm,fullpage,authblk,titlesec,secdot}

\usepackage{filecontents}



\newtheorem{theorem}{Theorem}
\newtheorem*{theorem*}{Theorem}
\newtheorem{lemma}{Lemma}

\newtheorem*{remark*}{Remark}

\begin{document}
\title{\uppercase{A Note on an Asymptotic Expansion related to the Dickman function}}
\author{C. S. Franze \thanks{franze.3@osu.edu}}
\affil{Department of Mathematics,\\ The Ohio State University}
\date{}
\maketitle
\abstract{In this paper we refine an asymptotic expansion given by Soundararajan \cite{Sound} for a family of multiple integrals related to the Dickman function. The result suggests a relatively simple approach to computing these integrals numerically.}

\section{Introduction}

The Dickman function, $\rho(u)$, satisfies the differential equation $\left(u\rho(u)\right)'=-\rho(u-1)$ for $u\ge1$, while $\rho(u)=1$ for $0\le u\le1$. Solving this differential equation one unit interval at a time, we find that
\begin{equation*}
  \rho(u)=\sum_{\ell=0}^{\infty}(-1)^{\ell}K_{\ell}(u),
\end{equation*}
where, for $\ell\ge1$,
\begin{equation*}
K_{\ell}(u):=\frac{1}{\ell!}\idotsint\limits_{\substack{t_{1},\ldots,t_{\ell}\ge1\\t_{1}+\cdots+t_{\ell}\le u}}\frac{dt_{1}}{t_{1}}\cdots\frac{dt_{\ell}}{t_{\ell}},
\end{equation*}
while $K_{0}(u):=1$ for $0\le u\le1$. Note that the sum above is actually finite since $K_{\ell}(u)=0$ if $u\le\ell$.

An equivalent form of this integral decomposition of $\rho(u)$ appears in Ramanujan's unpublished papers, and actually predates Dickman's published account. While the function itself has been thoroughly investigated by many researchers over the years, the integrals, $K_\ell(u)$, appearing above have not received as much attention. For instance, a procedure to numerically calculate $K_\ell(u)$ can be found in the work of Grupp and Richert \cite{Grupp} published in 1986, while a published account of the asymptotic behavior of $K_\ell(u)$ only recently appeared in the work of Soundararajan \cite{Sound} in 2012, in which he proved
\begin{theorem*}[Soundararajan, 2012]\label{MainSound}
For each integer $\ell\ge1$, as $u\rightarrow\infty$,
\begin{equation}\label{SoundEXP}
K_{\ell}(u)=\sum_{r=0}^{\ell}\frac{(-1)^{r}}{(\ell-r)!}C_{r}\log^{\ell-r}u+O_{\ell}\left(\frac{\log^{\ell}u}{u}\right),
\end{equation}
where the constants $C_{r}$ are given by the generating function,
\begin{equation}\label{DickmanGenFunc}
  \sum_{r=0}^{\infty}C_{r}z^{r}=\frac{e^{\gamma z}}{\Gamma(1-z)}.
\end{equation}
\end{theorem*}
This asymptotic expansion was first conjectured by Broadhurst \cite{Broadhurst} in 2010, where he considered a generalized class of polylogarithms. An expansion of a similiar shape appears in the recent work of Smith \cite{Smith}. The purpose of this paper is to refine the expansion above, providing more terms, and suggesting an alternative method for computing $K_\ell(u)$. Specifically, we prove 
\begin{theorem}\label{Main}
  For each integer $J\ge0$ and $\ell\ge1$, provided $u\ge\ell$,
\begin{equation}\label{FranzeEXP}
  K_{\ell}(u)=\sum_{j=0}^{J}\sum_{m=0}^{\ell}\sum_{r=0}^{\ell-m}\frac{(-1)^{r}}{m!(\ell-m-r)!}E_{j,m}C_{r}\left(\log^{\ell-m-r}u\right)^{(j)}+O_{J,\ell}\left(\frac{\log^{\ell}eu}{u^{J+1}}\right),
\end{equation}
where the constants $C_{r}$ are given above in \eqref{DickmanGenFunc}, and $E_{j,m}$ are given by the generating function
\begin{equation}\label{E_nm_Def}
\sum_{j=0}^{\infty}E_{j,m}z^{j}=\left(\int_{0}^{z}\frac{1-e^{-t}}{t}dt\right)^{m}.
\end{equation}
\end{theorem}

This expansion offers a relatively simple way to compute $K_\ell(u)$ when compared to that of Grupp and Richert. Their treatment uses many recursively defined power series expansions to approximate $K_\ell(u)$, whereas Theorem \ref{Main} provides a single series expansion. In addition, the expansion above converges faster with $u$, while the power series expansions of Grupp and Richert have a slow rate of convergence outside their centers. However, one would need good bounds on the implied constant appearing in the error term above to make this a legitimate numerical method for computing $K_\ell(u)$. We leave this as a challenge to future researchers.

\section{An Auxilliary Result}
We will deduce Theorem \ref{Main} from an asymptotic expansion for a related family of integrals, $K_\ell(u,\kappa)$, previously investigated by the author in \cite{Franze}. To elaborate on this family, for integers $\ell\ge1$, we define
\begin{equation}\label{Alt_Mult_Int_Rep}
K_{\ell}(u,\kappa):=\frac{1}{\ell!}\idotsint\limits_{\substack{t_{1},\ldots,t_{\ell}\ge1\\t_{1}+\cdots+t_{\ell}\le u}}\left(u-(t_1+\cdots+t_\ell)\right)^{\kappa}\frac{dt_{1}}{t_{1}}\cdots\frac{dt_{\ell}}{t_{\ell}},
\end{equation}
while $K_{0}(u,\kappa):=u^{\kappa}$ for $0\le u\le1$. While these integrals can be defined for all real $\kappa>-1$, we will restrict our attention to integer $\kappa\ge0$. If $\kappa=0$, then of course $K_{\ell}(u,0)=K_{\ell}(u)$. For integer $\kappa\ge1$, we have the identity
\begin{equation}\label{Convol_Identity}
K_{\ell}(u,\kappa)=\kappa \int_{\ell}^{u}\left(u-t\right)^{\kappa-1}K_{\ell}(t)\ dt.
\end{equation}
We also have the useful contour integral representation available for all integer $\kappa\ge0$, 
\begin{equation}\label{Contour_Int_Rep}
  K _{\ell}(u,\kappa)=\frac{\Gamma(\kappa+1)}{\ell!}\frac{u^{\kappa}}{2\pi i}\int_{c-i\infty}^{c+i\infty}e^{s}E_1\left(\frac{s}{u}\right)^{\ell}\frac{ds}{s^{\kappa+1}},
\end{equation}
where $c>0$, and
\begin{equation*}
  E_1(s):=\int_{1}^{\infty}e^{-st}\frac{dt}{t}.
\end{equation*}
The convolution identity in \eqref{Convol_Identity} was proved in \cite[see (17)]{Franze}, while the contour integral representation in \eqref{Contour_Int_Rep} was shown in \cite[see (23)]{Franze}. Using this last representation, the author established in \cite{Franze} an asymptotic expansion for $K_\ell(u,\kappa)$, given below. The proof follows that of \eqref{SoundEXP}, but requires some additional calculations.
\begin{theorem*}[Franze]\label{Gen_Sound_Theorem}
  For each integer $\kappa\ge0$, and $\ell\ge1$, provided $u\ge\ell$,
  \begin{equation*}
    K_{\ell}(u,\kappa)=\sum_{m=0}^{\ell}\sum_{n=m}^{\kappa}\sum_{r=0}^{\ell-m}\frac{(-1)^{r}\Gamma(\kappa+1)}{m!(\ell-m-r)!}E_{n,m}C_{r,\kappa-n}u^{\kappa-n}\log^{\ell-m-r}u+O_{\kappa,\ell}\left(\frac{\log^{\ell}eu}{u}\right),
  \end{equation*}
  where the constants $E_{n,m}$ are given above in \eqref{E_nm_Def}, and $C_{r,\kappa}$ are given by the generating function,
  \begin{equation}\label{genfunc_C_rk}
    \sum_{r=0}^{\infty}C_{r,\kappa}z^r=\frac{e^{\gamma z}}{\Gamma(\kappa+1-z)}.
  \end{equation}
\end{theorem*}
\begin{remark*}
  Observe that the $n$-sum vanishes if $m>\kappa$. We use this observation to avoid the repetitive condition, $0\le m\le\min(\kappa,\ell)$, on the $m$-sum.
\end{remark*}
For our purpose, we will need to make the error term above explicit, at least when $\kappa\ge\ell$. Thus, we provide a quick sketch of the proof under this assumption. 
\begin{proof}[Sketch of proof of Theorem \ref{Gen_Sound_Theorem}]
To begin, it was observed in \cite[p.28]{Sound} that
\begin{equation*}
E_{1}\left(\frac{s}{u}\right)=G(u,s)-\log\left(\frac{s}{u}\right)-\gamma,
\end{equation*}
where
\begin{equation}\label{G_Def}
G(u,s):=\int_{0}^{1/u}\frac{1-e^{-ts}}{t}dt.
\end{equation}
Inserting this relationship into \eqref{Contour_Int_Rep}, and then using the binomial theorem, gives
\begin{equation}\label{E14}
K_{\ell}(u,\kappa)=\sum_{m=0}^{\ell}\binom{\ell}{m}\frac{\Gamma(\kappa+1)}{\ell!}u^{\kappa}\frac{1}{2\pi i}\int_{c-i\infty}^{c+i\infty}e^{s}G(u,s)^{m}\left(\log u-\log s-\gamma\right)^{\ell-m}\frac{ds}{s^{\kappa+1}}.
\end{equation}
Next, we truncate the Taylor series expansion for $G(u,s)^{m}$, writing
\begin{equation}\label{R_head}
G(u,s)^{m}=\sum_{n=m}^{\kappa}E_{n,m}\left(\frac{s}{u}\right)^n+R_m(u,s),
\end{equation}
where
\begin{equation}\label{R_tail}
  R_m(u,s)=\sum_{n=\kappa+1}^{\infty}E_{n,m}\left(\frac{s}{u}\right)^{n}.
\end{equation}
An exact expression for the coefficients, $E_{n,m}$, is given in \cite[see (8)]{Franze}. Here, it is enough to know that $E_{n,m}=0$ if $m>n$, and the bound in \cite[see (30)]{Franze},
\begin{equation}\label{wasteful_E_bound}
  \left|E_{n,m}\right|\le\frac{m^n}{n!}.
\end{equation}
After substituting \eqref{R_head} into \eqref{E14}, we identify a main term and an error term,
\begin{equation}\label{MT_ET_Identity}
  K_\ell(u,\kappa)=\widetilde{K}_\ell(u,\kappa)+\mathscr{E}_\ell(u,\kappa),
\end{equation}
where
\begin{equation*}
\widetilde{K}_{\ell}(u,\kappa)=\sum_{m=0}^{\ell}\sum_{n=m}^{\kappa}\binom{\ell}{m}\frac{\Gamma(\kappa+1)}{\ell!}E_{n,m}u^{\kappa-n}\frac{1}{2\pi i}\int_{c-i\infty}^{c+i\infty}e^{s}s^n\left(\log u-\log s-\gamma\right)^{\ell-m}\frac{ds}{s^{\kappa+1}},
\end{equation*}
and,
\begin{equation}\label{Error_Term}
\mathscr{E}_\ell(u,\kappa)=\sum_{m=0}^{\ell}\binom{\ell}{m}\frac{\Gamma(\kappa+1)}{\ell!}\frac{u^{\kappa}}{2\pi i}\int_{c-i\infty}^{c+i\infty}e^{s}R_{m}(u,s)\left(\log u-\log s-\gamma\right)^{\ell-m}\frac{ds}{s^{\kappa+1}}.
\end{equation}
For the error term, $\mathscr{E}_\ell(u,\kappa)$, it was shown in \cite[Lemma 12]{Franze} that for $u\ge\ell$,
\begin{equation*}
  \mathscr{E}_{\ell}(u,\kappa)\ll_{\kappa,\ell}\frac{\log^{\ell}eu}{u}.
\end{equation*}
For the main term, $\widetilde{K}_{\ell}(u,\kappa)$, we use the binomial theorem once more,
\begin{equation*}
\left(\log u-\log s-\gamma\right)^{\ell-m}=\sum_{r=0}^{\ell-m}\binom{\ell-m}{r}(-1)^{r}\left(\log s+\gamma\right)^{r}\log^{\ell-m-r}u,
\end{equation*}
and define the constants $C_{r,\kappa}$ by
\begin{equation}\label{E1_begin}
C_{r,\kappa}:=\frac{1}{r!}\frac{1}{2\pi i}\int_{c-i\infty}^{c+i\infty}\frac{e^{s}}{s^{\kappa+1}}\left(\log s+\gamma\right)^{r}ds.
\end{equation}
It is then a straightforward exercise to show that $\widetilde{K}_\ell(u,\kappa)$ takes the form stated in the theorem,
\begin{equation}\label{AltFormMainK}
  \widetilde{K}_{\ell}(u,\kappa)=\sum_{m=0}^{\ell}\sum_{n=m}^{\kappa}\sum_{r=0}^{\ell-m}\frac{(-1)^{r}\Gamma(\kappa+1)}{m!(\ell-m-r)!}E_{n,m}C_{r,\kappa-n}u^{\kappa-n}\log^{\ell-m-r}u.
\end{equation}
See \cite[Lemma 4]{Franze} to verify that the constants $C_{r,\kappa}$ defined in \eqref{E1_begin} are generated by the function given above in \eqref{genfunc_C_rk}. This completes the sketch of the proof of Theorem \ref{Gen_Sound_Theorem}.
\end{proof}

Before continuing on, some comments about the constants $C_{r,\kappa}$ are in order. First, it is clear that $C_{r,\kappa}$ generalize $C_r$ since comparing \eqref{DickmanGenFunc} and \eqref{genfunc_C_rk}, we have $C_{r,0}=C_r$. Broadhurst called the constants, $C_r$, the \emph{Dickman constants} and conjectured their generating function, $e^{\gamma z}/\Gamma(1-z)$, though they are initially defined by integrals as in \eqref{E1_begin}. From this function one can easily deduce that for $r\ge1$, 
\begin{equation*}
  C_r=\frac{1}{r!}\sum_{k=1}^{r}(-1)^{k}B_{r,k}\left(0,1!\zeta(2),2!\zeta(3),\ldots,(r-k)!\zeta(r-k+1)\right),    
\end{equation*}  
where $B_{r,k}$ is a certain Bell polynomial, defined below in \eqref{BellY}. For integer $\kappa\ge1$, the \emph{Generalized Dickman constants}, $C_{r,\kappa}$, can be related back to $C_r$ using the recursive formula,
\begin{equation*}
  C_{r,\kappa}=\sum_{j=0}^{r}\frac{C_{j,\kappa-1}}{\kappa^{r-j+1}}.
\end{equation*}
Proofs of these observations can be found in \cite{Franze}. For our purpose, we provide a different recursive formula for $C_{r,\kappa}$, whose proof uses only the integral definition of $C_{r,\kappa}$ given above in \eqref{E1_begin}.
\begin{lemma}\label{Recursive1}
For natural numbers $r\ge0$ and $\kappa\ge 1$, we have
\begin{equation*}
\kappa\ C_{r,\kappa}=C_{r,\kappa-1}+C_{r-1,\kappa},
\end{equation*}
where we define $C_{-1,\kappa}:=0$.
\end{lemma}
\begin{proof}
Using integration by parts, we see that
\begin{align*}
  C_{r,\kappa-1}&=-\frac{1}{r!}\frac{1}{2\pi i}
            \int_{c-i\infty}^{c+i\infty}e^{s}\frac{d}{ds}\left(\frac{(\log s+\gamma)^{r}}{s^{\kappa}}\right)ds\\
            &=-\frac{1}{(r-1)!}\frac{1}{2\pi i}\int_{c-i\infty}^{c+i\infty}e^{s}\frac{(\log s+\gamma)^{r-1}}{s^{\kappa+1}}ds+\kappa\frac{1}{r!}\frac{1}{2\pi i}\int_{c-i\infty}^{c+i\infty}e^{s}\frac{(\log s+\gamma)^{r}}{s^{\kappa+1}}ds,
\end{align*}
and thus,
    \begin{equation*}
        C_{r,\kappa-1}=-C_{r-1,\kappa}+\kappa C_{r,\kappa}.
    \end{equation*}
\end{proof}
\section{Proof of Theorem 1}
Now that we have made the main term and error term in \eqref{MT_ET_Identity} explicit, we are ready to prove Theorem \ref{Main}. We intend to show that the expansion given in Theorem \ref{Main} is obtained by repeated differentiation of \eqref{MT_ET_Identity} in the next several lemmas. To begin, observe that
\begin{lemma}\label{L_diff}
For each integer $\kappa\ge1$,
\begin{equation}\label{K_deriv}
\frac{d}{du}\left(K_{\ell}(u,\kappa)\right)=\kappa K_{\ell}(u,\kappa-1).
\end{equation}
\end{lemma}
\begin{proof}
  This is an application of differentiation under the integral sign, using equation \eqref{Convol_Identity}.
\end{proof}
On the other hand, the derivative of the main term, $\widetilde{K}_\ell(u,\kappa)$, obeys
\begin{lemma}\label{DiffLemma}
For each integer $\kappa\ge1$, and $\widetilde{K}_{\ell}(u,\kappa)$ defined as above, we have
\begin{equation*}
\frac{d}{du}\left(\widetilde{K}_{\ell}(u,\kappa)\right)=\kappa \widetilde{K}_{\ell}(u,\kappa-1)+\sum_{m=0}^{\ell}\sum_{r=0}^{\ell-m}\frac{(-1)^{r}\Gamma(\kappa+1)}{m!(\ell-m-r)!}E_{\kappa,m}C_{r,0}\frac{d}{du}\left(\log^{\ell-m-r}u\right).
\end{equation*}
\end{lemma}
\begin{proof}
Differentiating \eqref{AltFormMainK}, and separating the terms with $n=\kappa$, we may see, upon re-indexing, that
\begin{multline*}
  \frac{d}{du}\left(\widetilde{K}_{\ell}(u,\kappa)\right)=\sum_{m=0}^{\ell}\sum_{n=m}^{\kappa-1}\sum_{r=0}^{\ell-m}
  \frac{(-1)^{r}\Gamma(\kappa+1)}{m!(\ell-m-r)!}E_{n,m}\left(\kappa C_{r,\kappa}-C_{r-1,\kappa}\right)u^{\kappa-1-n}\log^{\ell-m-r}u\\
  +\sum_{m=0}^{\ell}\sum_{r=0}^{\ell-m}\frac{(-1)^{r}\Gamma(\kappa+1)}{m!(\ell-m-r)!}E_{\kappa,m}C_{r,0}\frac{d}{du}\left(\log^{\ell-m-r}u\right).
\end{multline*}
Using Lemma \ref{Recursive1} in the first sum concludes the proof of this lemma.
\end{proof}
For higher order derivatives of $\widetilde{K}_\ell(u,\kappa)$, we have
\begin{lemma}\label{DiffLemma2}
  For each $1\le \nu\le \kappa$, we have
  \begin{equation*}
    \widetilde{K}_{\ell}^{(\nu)}(u,\kappa)=(\kappa)_{\nu}\widetilde{K}_{\ell}(u,\kappa-\nu)
        +\sum_{j=0}^{\nu-1}\sum_{m=0}^{\ell}\sum_{r=0}^{\ell-m}
            \frac{(-1)^{r}\Gamma(\kappa+1)}{m!(\ell-m-r)!}E_{\kappa-j,m}C_{r,0}\left(\log^{\ell-m-r}u\right)^{(\nu-j)},
  \end{equation*}
  where all derivatives are taken with respect to $u$, and $(x)_{n}$ denotes the falling factorial,
  \begin{equation*}
  (x)_{n}:=x(x-1)(x-2)\cdots(x-n+1).
\end{equation*}
\end{lemma}
\begin{proof}
  When $\nu=1$, Lemma \ref{DiffLemma2} can easily be deduced from Lemma \ref{DiffLemma}. Proceeding by induction, suppose the theorem is true for $\nu-1$. Since
  \begin{equation*}
    \widetilde{K}_{\ell}^{(\nu)}(u,\kappa)=\frac{d}{du}\widetilde{K}_{\ell}^{(\nu-1)}(u,\kappa),
  \end{equation*}
  we have
  \begin{equation*}
    \widetilde{K}_{\ell}^{(\nu)}(u,\kappa)=(\kappa)_{\nu-1}\widetilde{K}_{\ell}^{'}(u,\kappa-\nu+1)
        +\sum_{j=0}^{\nu-2}\sum_{m=0}^{\ell}\sum_{r=0}^{\ell-m}
            \frac{(-1)^{r}\Gamma(\kappa+1)}{m!(\ell-m-r)!}E_{\kappa-j,m}C_{r,0}\left(\log^{\ell-m-r}u\right)^{(\nu-j)}.
  \end{equation*}
  Now, using Lemma \ref{DiffLemma} on $\widetilde{K}_{\ell}^{'}(u,\kappa-\nu+1)$, this reads
  \begin{multline*}
    \widetilde{K}_{\ell}^{(\nu)}(u,\kappa)=(\kappa)_{\nu-1}(\kappa-\nu+1)\widetilde{K}_{\ell}(u,\kappa-\nu)\\
        +\sum_{m=0}^{\ell}\sum_{r=0}^{\ell-m}
            \frac{(-1)^{r}(\kappa)_{\nu-1}\Gamma(\kappa-\nu+2)}{m!(\ell-m-r)!}E_{\kappa-\nu+1,m}C_{r,0}\left(\log^{\ell-m-r}u\right)^{(1)}\\
        +\sum_{j=0}^{\nu-2}\sum_{m=0}^{\ell}\sum_{r=0}^{\ell-m}
            \frac{(-1)^{r}\Gamma(\kappa+1)}{m!(\ell-m-r)!}E_{\kappa-j,m}C_{r,0}\left(\log^{\ell-m-r}u\right)^{(\nu-j)}.
  \end{multline*}
  Since $(\kappa)_{\nu-1}(\kappa-\nu+1)=(\kappa)_{\nu}$, and $(\kappa)_{\nu-1}\Gamma(\kappa-\nu+2)=\Gamma(\kappa+1)$, the first sum can be absorbed into the last sum as the term corresponding to $j=\nu-1$, we are done.
\end{proof}
Finally, for the derivatives of the error term, $\mathscr{E}_{\ell}(u,\kappa)$, we prove
\begin{lemma}\label{L_Error_Derivatives}
  If $\ell\ge1$, $\kappa\ge0$, and $0\le\nu\le\kappa$, then for $u\ge\ell$,
  \begin{equation*}
    \mathscr{E}_{\ell}^{(\nu)}(u,\kappa)\ll_{\nu,\ell,\kappa}\frac{\log^{\ell}eu}{u^{\nu+1}}.
  \end{equation*}
\end{lemma}
\begin{proof}
First, integrate \eqref{Error_Term} by parts so that
\begin{equation}\label{EinterF}
  \mathscr{E}_{\ell}(u,\kappa)=\frac{\kappa!}{\ell!}\sum_{m=0}^{\ell}\binom{\ell}{m}\left(\mathcal{I}_{1}+\mathcal{I}_{2}+\mathcal{I}_{3}\right),
\end{equation}
where,
\begin{align*}
  \mathcal{I}_{1}&=\frac{u^{\kappa}}{2\pi i}\int_{(c)}e^{s}\frac{L\left(u,s\right)^{\ell-m-1}}{s}\left(\frac{(\ell-m)R_{m}(u,s)}{s^{\kappa+1}}\right)ds,\\
  \mathcal{I}_{2}&=\frac{u^{\kappa}}{2\pi i}\int_{(c)}e^{s}\frac{L\left(u,s\right)^{\ell-m}}{s}\left(\frac{-s\frac{\partial}{\partial s}R_{m}(u,s)}{s^{\kappa+1}}\right)ds,\\
  \mathcal{I}_{3}&=\frac{u^{\kappa}}{2\pi i}\int_{(c)}e^{s}\frac{L\left(u,s\right)^{\ell-m}}{s}\left(\frac{(\kappa+1)R_{m}(u,s)}{s^{\kappa+1}}\right)ds,
\end{align*}
and we have abbreviated $\int_{(c)}:=\int_{c-i\infty}^{c+i\infty}$, as well as
\begin{equation}\label{L_abbrev}
  L(u,s):=\log u-\log s-\gamma.
\end{equation}
Differentiating equation \eqref{EinterF} with respect to $u$ then gives
\begin{equation}\label{Einter}
  \mathscr{E}_{\ell}^{(\nu)}(u,\kappa)=\frac{\kappa!}{\ell!}\sum_{m=0}^{\ell}\binom{\ell}{m}\left(\mathcal{I}_{1}^{(\nu)}+\mathcal{I}_{2}^{(\nu)}+\mathcal{I}_{3}^{(\nu)}\right).
\end{equation}
To clear any confusion, all derivatives will be taken with respect to $u$, unless otherwise stated. The bound of Lemma \ref{L_Error_Derivatives} then follows if
\begin{equation*}
  \mathcal{I}^{(\nu)}_1, \mathcal{I}^{(\nu)}_2, \mathcal{I}^{(\nu)}_3\ll_{\kappa,\ell,\nu,m}\frac{\log^\ell eu}{u^{\nu+1}},
\end{equation*} 
upon inserting these bounds into equation \eqref{Einter} and summing over $m$.

It is worth pointing out that all of these integrals vanish when $m=0$ since $R_m=0$ in that case. Here we consider only $\mathcal{I}_{1}^{(\nu)}$ since the other two integrals are handled similarly. Also, observe that $\mathcal{I}_{1}^{(\nu)}$ does not appear unless $\ell\ge2$. Now, using Leibniz's formula twice, we have
\begin{equation}\label{I1_Leibniz}
  \mathcal{I}_{1}^{(\nu)}=\sum_{n_1=0}^{\nu}\sum_{n_2=0}^{n_1}\binom{\nu}{n_1}\binom{n_1}{n_2}\left(u^{\kappa}\right)^{(\nu-n_1)}\mathcal{J}_{1}(n_1,n_2),
\end{equation}
where
\begin{equation}\label{JmixedDer}
  \mathcal{J}_{1}(n_1,n_2):=\frac{(\ell-m)}{2\pi i}\int_{(c)}e^s\left(L\left(u,s\right)^{\ell-m-1}\right)^{(n_1-n_2)}\left(R_m(u,s)\right)^{(n_2)}\frac{ds}{s^{\kappa+2}}.
\end{equation}
Next, we split the integral appearing in $\mathcal{J}_{1}(n_1,n_2)$ into two parts, $\Gamma_1$ and $\Gamma_2$, where
\begin{align*}
  \Gamma_1&:=\left\{s\in\mathbb{C}:\Re s=1,1\le|s|<u\right\},\\
  \Gamma_2&:=\left\{s\in\mathbb{C}:\Re s=1,|s|>u\right\}.
\end{align*}
Lemma \ref{L_der_lemma} and Lemma \ref{R_mDerLemma} then give the bound 
\begin{equation*}
  \left|\mathcal{J}_{1}(n_1,n_2)\right|\ll\int_{\Gamma_1}\frac{\left(1+\log u\right)^{\ell-m-1}}{u^{n_1-n_2}}\frac{|s|^{\kappa+1}}{u^{\kappa+1+n_2}}\frac{|ds|}{|s|^{\kappa+2}}+\int_{\Gamma_2}\frac{\left(1+\log|s|\right)^{\ell-m-1}}{u^{n_1-n_2}}\frac{|s|^{\kappa+1}}{u^{\kappa+n_2}}\frac{|ds|}{|s|^{\kappa+2}},
\end{equation*}
where the implied constant depends at most on $\kappa,\ell,n_1,n_2$ and $m$. This immediately implies that
\begin{equation}\label{intermediateJbound}
  \left|\mathcal{J}_{1}(n_1,n_2)\right|\ll_{\kappa,\ell,n_1,n_2,m}\frac{\log^{\ell} eu}{u^{\kappa+1+n_1}},
\end{equation}
since, by \cite[Lemma 9]{Franze} for example,
\begin{equation*}
  \int_{\Gamma_1}\frac{|ds|}{|s|}\ll \left(1+\log u\right),
\end{equation*}
and,
\begin{equation*}
  \int_{\Gamma_2}\frac{\left(1+\log|s|\right)^{\ell-m-1}}{|s|^{2}}|ds|\ll \frac{\left(1+\log u\right)^{\ell-m-1}}{u}.
\end{equation*}
Substituting the bound in \eqref{intermediateJbound} into equation \eqref{I1_Leibniz} and summing over $n_1$ and $n_2$, we find that
\begin{equation*}
  \mathcal{I}_{1}^{(\nu)}\ll_{\kappa,\ell,\nu,m}\sum_{n_1=1}^{\nu}\sum_{n_2=0}^{n_1}u^{\kappa-\nu+n_1}\left(\frac{\log^{\ell}eu}{u^{\kappa+1+n_1}}\right)\ll_{\kappa,\ell,\nu,m}\frac{\log^{\ell}eu}{u^{\nu+1}},
\end{equation*}
completing the proof. As remarked earlier, the bound for $\mathcal{I}^{(\nu)}_2$ and $\mathcal{I}^{(\nu)}_3$ follows similarly, although Lemma \ref{I_2_Lemma} is additionally needed for $\mathcal{I}^{(\nu)}_2$.
\end{proof}
The proof of Theorem \ref{Main} now easily follows.
\begin{proof}[Proof of Theorem \ref{Main}]
Setting $\nu=\kappa$ in Lemma \ref{DiffLemma2}, we find that
\begin{equation}\label{Lemma4REMARK}
  \widetilde{K}_{\ell}^{(\kappa)}(u,\kappa)=\sum_{j=0}^{\kappa}\sum_{m=0}^{\ell}\sum_{r=0}^{\ell-m}
    \frac{(-1)^{r}\Gamma(\kappa+1)}{m!(\ell-m-r)!}E_{\kappa-j,m}C_{r,0}\left(\log^{\ell-m-r}u\right)^{(\kappa-j)},
\end{equation}
since $(\kappa)_{\kappa}\widetilde{K}_{\ell}(u,0)$ corresponds to the term $j=\kappa$. Differentiating both sides of \eqref{MT_ET_Identity} gives
\begin{align}\label{derKMTET}
  \frac{K_{\ell}^{(\kappa)}(u,\kappa)}{\kappa!}&=\frac{\widetilde{K}_{\ell}^{(\kappa)}(u,\kappa)}{\kappa!}+\frac{\mathscr{E}_{\ell}^{(\kappa)}(u,\kappa)}{\kappa!}.
\end{align}
The left-hand side is $K_{\ell}(u)$, by repeated use of Lemma \ref{L_diff}. Thus, substituting \eqref{Lemma4REMARK} into \eqref{derKMTET},
\begin{equation*}
  K_{\ell}(u)=\sum_{j=0}^{\kappa}\sum_{m=0}^{\ell}\sum_{r=0}^{\ell-m}
    \frac{(-1)^{r}}{m!(\ell-m-r)!}E_{\kappa-j,m}C_{r,0}\left(\log^{\ell-m-r}u\right)^{(\kappa-j)}+\frac{\mathscr{E}_{\ell}^{(\kappa)}(u,\kappa)}{\kappa!}.
\end{equation*}
Using Lemma \ref{L_Error_Derivatives} and re-indexing the sum on $j$ then gives
\begin{equation*}
  K_{\ell}(u)=\sum_{j=0}^{\kappa}\sum_{m=0}^{\ell}\sum_{r=0}^{\ell-m}
    \frac{(-1)^{r}}{m!(\ell-m-r)!}E_{j,m}C_{r,0}\left(\log^{\ell-m-r}u\right)^{(j)}+O_{\kappa,\ell}\left(\frac{\log^{\ell}eu}{u^{\kappa+1}}\right).
\end{equation*}
Thus, the proof of Theorem \ref{Main} is complete upon replacing $\kappa$ with $J$. 
\end{proof}

The observant reader will notice that we have proved Theorem \ref{Main} only for $J\ge\ell$ since \eqref{MT_ET_Identity} was dependent upon this assumption. However, this is enough since the main term appearing in Theorem \ref{Main} can be truncated to include the remaining cases where $0\le J<\ell$ using the observation that the $j$-th term of the sum in Theorem \ref{Main} is $O\left(\frac{\log^{\ell}eu}{u^{j}}\right)$.

\section{Lemmata}
The results of this section are used to prove Lemma \ref{L_Error_Derivatives}. Throughout this section we will make extensive use of Fa\`a di Bruno's formula,
\begin{equation}\label{FDB}
  \frac{d^{n}}{du^{n}}f\left(g(u))\right)=\sum_{k=1}^{n}f^{(k)}\left(g(u)\right)B_{n,k}\left(g'(x),\ldots,g^{(n-k+1)}(x)\right),
\end{equation}
where $B_{n,k}$ denotes the Bell polynomial,
\begin{equation}\label{BellY}
  B_{n,k}\left(x_1,\ldots,x_{n-k+1}\right):=\sum_{\vec{j}\in\mathcal{S}_{n,k}}\frac{n!}{j_1!\cdots j_{n-k+1}!}\prod_{i=1}^{n-k+1}\left(\frac{x_i}{i!}\right)^{j_i},
\end{equation}
and the sum is taken over the elements of
\begin{equation}\label{BellY_Set}
  \mathcal{S}_{n,k}:=\left\{(j_1,\ldots,j_{n-k+1})\in\mathbb{Z}^{n-k+1}:j_i\ge0,\sum_{i=1}^{n-k+1}i\ j_{i}=n,\sum_{i=1}^{n-k+1}j_{i}=k\right\}.
\end{equation}
To begin, we use this formula to establish a bound on the derivatives of the function $L(u,s)$ defined in \eqref{L_abbrev}.
\begin{lemma}\label{L_der_lemma}
  Suppose that $\Re s=1$, and that $u\ge1$ . For integers $\nu\ge0$ and $\ell\ge1$,
  \begin{equation*}
    \left|\frac{\partial^{\nu}}{\partial u^{\nu}}L(u,s)^{\ell}\right|\ll_{\nu,\ell}
    \begin{cases}
      \dfrac{\left(1+\log u\right)^{\ell}}{u^{\nu}},& \text{if $|s|<u$,}\\
      \dfrac{\left(1+\log |s|\right)^{\ell}}{u^{\nu}},& \text{if $|s|>u$.}
    \end{cases}
  \end{equation*}
\end{lemma}
\begin{proof}
  If $\nu=0$, the bound follows immediately from \cite[Lemma 10]{Franze},
  \begin{equation}\label{Old_L_bound}
    \left|L(u,s)\right|\ll
    \begin{cases}
      1+\log u,& \text{if $|s|<u$,}\\
      1+\log|s|,& \text{if $|s|>u$.}
    \end{cases}
  \end{equation}
  If $\nu\ge1$, an application of \eqref{FDB} gives
  \begin{equation*}
    \frac{\partial^{\nu}}{\partial u^{\nu}}L(u,s)^{\ell}=\sum_{j=1}^{\nu}(\ell)_{j}L(u,s)^{\ell-j}B_{\nu,j}\left(\frac{\partial}{\partial u}L(u,s),\ldots,\frac{\partial^{\nu-j+1}}{\partial u^{\nu-j+1}}L(u,s)\right).
  \end{equation*}
  Using \eqref{BellY}, the Bell polynomial, $B_{\nu,j}$, appearing above is equal to $D_{\nu,j}/u^{\nu}$ for some constant $D_{\nu,j}$. Making this substitution, the bound then follows easily from \eqref{Old_L_bound}.
\end{proof}
Next, we establish some bounds involving the function $G(u,s)$ defined in \eqref{G_Def}. For example, it was shown in \cite{Franze} that $G(u,s)$ satisfies
\begin{lemma}\label{SimpleG_Lemma}
  Suppose that $\Re s=1$, and that $u\ge1$. Then we have
  \begin{equation*}
    |G(u,s)|\ll
    \begin{cases}
      \frac{|s|}{u}, &\text{if $|s|<u$},\\
      1+\log \left(\frac{|s|}{u}\right), &\text{if $|s|>u$.}
    \end{cases}
  \end{equation*}
\end{lemma}
\begin{proof}
  The bound can easily be deduced using Soundararajan's estimate \cite[p.29]{Sound},
  \begin{equation*}
    |G(u,s)|\ll\int_{0}^{1/u}\min\left(|s|,\frac{1}{t}\right)dt.
  \end{equation*}
  See \cite[Lemma 5]{Franze} for more detail.
\end{proof}
Moreover, for derivatives of $G(u,s)$ we have
\begin{lemma}\label{SimpleGDer_Lemma}
  Suppose that $\Re s=1$, $u\ge1$, and that $|s|>u$. For integers $\nu\ge1$,
  \begin{equation*}
    \left|\frac{\partial^{\nu}}{\partial u^{\nu}}G(u,s)\right|\ll_{\nu}\frac{|s|^{\nu-1}}{u^{2\nu-1}}
  \end{equation*}
\end{lemma}
\begin{proof}
  If $\nu=1$, differentiating \eqref{G_Def} gives
  \begin{equation*}
    \frac{\partial}{\partial u}G(u,s)=\frac{e^{-s/u}-1}{u},
  \end{equation*}
  and the bound follows since $|e^{-s/u}|\le1$. If $\nu\ge2$, applying the Leibniz rule gives
  \begin{equation*}
    \frac{\partial^{\nu}}{\partial u^{\nu}}G(u,s)=\sum_{n=1}^{\nu-1}\binom{\nu-1}{n}\frac{\partial^{n}}{\partial u^{n}}\left(e^{-s/u}\right)\frac{\partial^{\nu-1-n}}{\partial u^{\nu-1-n}}\left(\frac{1}{u}\right)+\left(e^{-s/u}-1\right)\frac{\partial^{\nu-1}}{\partial u^{\nu-1}}\left(\frac{1}{u}\right).
  \end{equation*}
  The bound stated in the lemma follows since
  \begin{equation}\label{Easy_u_der}
    \left|\frac{\partial^{\nu-1-n}}{\partial u^{\nu-1-n}}\left(\frac{1}{u}\right)\right|\ll_{\nu,n} \frac{1}{u^{\nu-n}}.
  \end{equation}
  and
  \begin{equation}\label{Easy_exp_der}
    \left|\frac{\partial^{n}}{\partial u^{n}}\left(e^{-s/u}\right)\right|\ll_{n}\frac{|s|^{n}}{u^{2n}}.
  \end{equation}
  While \eqref{Easy_u_der} is obvious, \eqref{Easy_exp_der} can be established by \eqref{FDB} since 
  \begin{equation}\label{Exp_Der_1}
    \frac{\partial^{n}}{\partial u^{n}}\left(e^{-s/u}\right)=\sum_{k=1}^{n}e^{-s/u}B_{n,k}\left(\frac{s}{u^2},-\frac{2!s}{u^3},\ldots,(-1)^{n-k}\frac{(n-k+1)!s}{u^{n-k+2}}\right),
  \end{equation}
  and, recalling \eqref{BellY} and \eqref{BellY_Set},
  \begin{equation}\label{Exp_Der_2}
    \left|B_{n,k}\left(\frac{s}{u^2},-\frac{2!s}{u^3},\ldots,(-1)^{n-k}\frac{(n-k+1)!s}{u^{n-k+2}}\right)\right|\ll_{n}\sum_{\vec{j}\in\mathcal{S}_{n,k}}\prod_{i=1}^{n-k+1}\left(\frac{s}{u^{i+1}}\right)^{j_{i}}\ll_{n,k}\frac{|s|^{k}}{u^{k+n}}.
  \end{equation}
\end{proof}
For the derivatives of $G(u,s)^{m}$, we have
\begin{lemma}\label{GPowerDer_Lemma}
  Suppose that $\Re s=1$, $u\ge1$, and that $|s|>u$. For integers $m,\nu\ge1$,
  \begin{equation*}
    \frac{\partial^\nu}{\partial u^\nu}\left(G(u,s)^m\right)\ll_{m,\nu}\dfrac{\left(1+\log|s|\right)^{m-1}}{u^\nu}\left(\frac{|s|}{u}\right)^{\nu-1}.
  \end{equation*}
\end{lemma}
\begin{proof}
  Applying equation \eqref{FDB}, it is apparent that we must bound
  \begin{equation}\label{NBD1}
    \sum_{k=1}^{\nu}(m)_{k}G(u,s)^{m-k}B_{\nu,k}\left(\frac{\partial}{\partial u}G(u,s),\frac{\partial^{2}}{\partial u^{2}}G(u,s),\ldots,\frac{\partial^{\nu-k+1}}{\partial u^{\nu-k+1}}G(u,s)\right).
  \end{equation}
  Using Lemma \ref{SimpleGDer_Lemma} and recalling \eqref{BellY_Set}, the Bell polynomial above is bounded by
  \begin{equation}\label{NBD2}
    \left|B_{\nu,k}\left(\frac{\partial}{\partial u}G,\frac{\partial^{2}}{\partial u^{2}}G,\ldots,\frac{\partial^{\nu-k+1}}{\partial u^{\nu-k+1}}G\right)\right|\ll_{\nu,k}\sum_{\vec{j}\in\mathcal{S}_{\nu,k}}\prod_{i=1}^{\nu-k+1}\left(\frac{|s|^{i-1}}{u^{2i-1}}\right)^{j_i}\ll_{\nu,k}\frac{|s|^{\nu-k}}{u^{2\nu-k}},
  \end{equation}
  where we have abbreviated $G:=G(u,s)$. The proof is concluded by combining \eqref{NBD1}, \eqref{NBD2}, and Lemma \ref{SimpleG_Lemma}.
\end{proof}
On the other hand, for the derivatives of $\frac{\partial}{\partial s}G(u,s)$, we have
\begin{lemma}\label{GMixedDer_Lemma}
  Suppose $\Re s=1$, $u\ge1$, and that $|s|>u$. For integers $\nu\ge0$,
  \begin{equation*}
    \left|\frac{\partial^{\nu}}{\partial u^{\nu}}\left(s\frac{\partial}{\partial s}G(u,s)\right)\right|\ll_{\nu}\dfrac{|s|^{\nu}}{u^{2\nu}}.
  \end{equation*}
\end{lemma}
\begin{proof}
  If $\nu=0$, the bound in the lemma is obvious since
  \begin{equation*}
    s \frac{\partial}{\partial s}G(u,s)=1-e^{-s/u}.
  \end{equation*}
  If $\nu\ge1$, the bound follows by \eqref{Easy_exp_der} since
  \begin{equation*}
    \left|\frac{\partial^{\nu}}{\partial u^{\nu}}\left(s\frac{\partial}{\partial s}G(u,s)\right)\right|=\left|\frac{\partial^{\nu}}{\partial u^{\nu}}\left(e^{-s/u}\right)\right|.
  \end{equation*}
\end{proof}
Finally, we establish bounds on the derivatives of $R_m(u,s)$ and $\frac{\partial}{\partial s}R_m(u,s)$.
\begin{lemma}\label{R_mDerLemma}
  Suppose that $\Re s=1$, and that $u\ge1$. For integers $1\le m\le\ell$ and $0\le\nu\le\kappa$,
  \begin{equation*}
    \left|\frac{\partial^\nu}{\partial u^\nu}R_m(u,s)\right|\ll_{m,\nu,\kappa}
    \begin{cases}
      \dfrac{|s|^{\kappa+1}}{u^{\kappa+1+\nu}}&\text{if $|s|<u$},\\
      \dfrac{|s|^{\kappa}}{u^{\kappa+\nu}}&\text{if $|s|>u$.}
    \end{cases}
  \end{equation*}
\end{lemma}
\begin{proof}
  When $|s|<u$ we use equation \eqref{R_tail} to observe that
  \begin{equation*}
    \frac{\partial^{\nu}}{\partial u^{\nu}}R_m(u,s)=\sum_{j=\kappa+1}^{\infty}(-1)^{\nu}j^{(\nu)}E_{j,m}\frac{s^{j}}{u^{j+\nu}},
  \end{equation*}
  where, for the rest of the manuscript, $j^{(\nu)}$ denotes the rising factorial,
  \begin{equation}\label{rising_factorial}
    j^{(\nu)}:=j(j+1)\cdots(j+\nu-1).
  \end{equation}
  The bound follows since the series,
  \begin{equation*}
    \sum_{j=\kappa+1}^{\infty}j^{(\nu)}|E_{j,m}|,
  \end{equation*}
  converges by the ratio test, using \eqref{wasteful_E_bound}. When $|s|>u$ we use equation \eqref{R_head}, instead, to observe that
  \begin{equation*}
    \frac{\partial^{\nu}}{\partial u^{\nu}}R_m(u,s)=\frac{\partial^{\nu}}{\partial u^{\nu}}\left(G(u,s)^{m}\right)-\sum_{j=m}^{\kappa}(-1)^{\nu}j^{(\nu)}E_{j,m}\frac{s^{j}}{u^{j+\nu}}.
  \end{equation*}
  To conclude the proof, we invoke Lemma \ref{SimpleG_Lemma} if $\nu=0$, and Lemma \ref{SimpleGDer_Lemma}, otherwise. 
\end{proof}
\begin{lemma}\label{I_2_Lemma}
  Suppose that $\Re s=1$, and that $u\ge1$. For integers $1\le m\le\ell$ and $0\le\nu\le\kappa$,
  \begin{equation*}
    \left|\frac{\partial^\nu}{\partial u^\nu}\left(s\frac{\partial}{\partial s}R_m(u,s)\right)\right|\ll_{m,\nu,\kappa}
    \begin{cases}
      \dfrac{|s|^{\kappa+1}}{u^{\kappa+1+\nu}}&\text{if $|s|<u$},\\
      \dfrac{|s|^{\kappa}}{u^{\kappa+\nu}}\max\left(1,(1+\log|s|)^{m-2}\right)&\text{if $|s|>u$.}
    \end{cases}
  \end{equation*}
\end{lemma}
\begin{proof}
  When $|s|<u$ we use equation \eqref{R_tail} to observe that
\begin{align*}
  \frac{\partial^{\nu}}{\partial u^{\nu}}\left(s\frac{\partial}{\partial s}R_m(u,s)\right)=\sum_{j=\kappa+1}^{\infty}(-1)^{\nu}j^{(\nu)}E_{j,m}\ j\frac{s^{j}}{u^{j+\nu}}.
\end{align*}
The bound follows since the series,
\begin{equation*}
  \sum_{j=\kappa+1}^{\infty}j^{(\nu)}|E_{j,m}|j,
\end{equation*}
converges by the ratio test, again using \eqref{wasteful_E_bound}.

On the other hand, when $|s|>u$, we apply the Leibniz rule to \eqref{R_head} so that
\begin{multline*}
  \frac{\partial^{\nu}}{\partial u^{\nu}}\left(s \frac{\partial}{\partial s}R_m(u,s)\right)=\sum_{j=0}^{\nu}\binom{\nu}{j}\frac{\partial^{\nu-j}}{\partial u^{\nu-j}}\left(s \frac{\partial}{\partial s}G(u,s)\right)\frac{\partial^{j}}{\partial u^{j}}\left(m G(u,s)^{m-1}\right)\\-\sum_{j=m}^{\kappa}(-1)^{\nu}j^{(\nu)}E_{j,m}\ j\frac{s^j}{u^{j+\nu}}.
\end{multline*}
Using Lemma \ref{GPowerDer_Lemma} and Lemma \ref{GMixedDer_Lemma} on the first sum and trivially bounding the second sum concludes the proof.
\end{proof}
\section{Concluding Remarks}
Although our main interest in this paper is $K_\ell(u)$, it is possible to further generalize our asymptotic to $K_\ell(u,\kappa)$ with very little effort. These integrals arise when one investigates the mean value of the generalized divisor function, $d_\kappa(n)$, over the smooth numbers, $S(x,y)$. For instance, letting $u=\log x/\log y$, it was shown in \cite{Franze} that as $x,y\rightarrow\infty$ with $u$ bounded,
\begin{equation*}
  \sum_{n\in S(x,y)}d_k(n)\sim\rho_\kappa(u)\ x\log^{\kappa-1}y,
\end{equation*}
where
\begin{equation*}
  \rho_\kappa(u)=\sum_{0\le\ell<u}\frac{(-\kappa)^{\ell}}{(\kappa-1)!}K_\ell(u,\kappa-1).
\end{equation*}
The Dickman function corresponds to the case $\kappa=1$ above. Thus, we take this opportunity to prove
\begin{theorem}\label{Secondary}
  For each integer $\kappa,J\ge0$ and $\ell\ge1$, provided $u\ge\ell$,
  \begin{multline*}
    K_{\ell}(u,\kappa)=\sum_{m=0}^{\ell}\sum_{n=m}^{\kappa}\sum_{r=0}^{\ell-m}\frac{(-1)^{r}\kappa!E_{n,m}}{m!(\ell-m-r)!}C_{r,\kappa-n}u^{\kappa-n}\log^{\ell-m-r}u\\
    +\sum_{j=1}^{J}\sum_{m=0}^{\ell}\sum_{r=0}^{\ell-m}\frac{(-1)^{r}\kappa!E_{\kappa+j,m}}{m!(\ell-m-r)!}C_{r}\left(\log^{\ell-m-r}u\right)^{(j)}+O_{\kappa,\ell,J}\left(\frac{\log^{\ell}eu}{u^{J+1}}\right).
  \end{multline*}
\end{theorem}
\begin{proof}
  Replace $\kappa$ by $\kappa+\nu$ in \eqref{MT_ET_Identity}, so that
  \begin{equation}\label{GenKExpansion}
    \frac{K_{\ell}^{(\nu)}(u,\kappa+\nu)}{(\kappa+\nu)_\nu}=\frac{\widetilde{K}_{\ell}^{(\nu)}(u,\kappa+\nu)}{(\kappa+\nu)_\nu}+\frac{\mathscr{E}_{\ell}^{(\nu)}(u,\kappa+\nu)}{(\kappa+\nu)_\nu}.
  \end{equation}
  Repeated use of Lemma \ref{L_diff} shows that the left-hand side is $K_\ell(u,\kappa)$, while Lemma \ref{L_Error_Derivatives} shows that
  \begin{equation*}
    \mathscr{E}_{\ell}^{(\nu)}(u,\kappa+\nu)\ll_{\kappa,\ell,\nu}\frac{\log^{\ell}eu}{u^{\nu+1}}.
  \end{equation*}
  Therefore, after applying Lemma \ref{DiffLemma2} to $\widetilde{K}_{\ell}^{(\nu)}(u,\kappa+\nu)/(\kappa+\nu)_\nu$, equation \eqref{GenKExpansion} yields
  \begin{equation*}
    K_\ell(u,\kappa)=\widetilde{K}_\ell(u,\kappa)+\sum_{j=1}^{\nu}\sum_{m=0}^{\ell}\sum_{r=0}^{\ell-m}\frac{(-1)^{r}\kappa!E_{\kappa+j,m}}{m!(\ell-m-r)!}C_r \left(\log^{\ell-m-r} u\right)^{(j)}+O_{\kappa,\ell,\nu}\left(\frac{\log^{\ell} eu}{u^{\nu+1}}\right).
  \end{equation*}
  Substituting the expression in \eqref{AltFormMainK} for $\widetilde{K}_\ell(u,\kappa)$ and letting $\nu=J$ then completes the proof.
\end{proof}
\section*{Acknowledgement}
I would like to thank the American Institute of Mathematics for providing me the opportunity to participate in the workshop on Bounded gaps between primes, and fostering the lively intellectual atmosphere that inspired this manuscript.
\bibliographystyle{abbrv}
\bibliography{\jobname} 

\begin{thebibliography}{1}

\bibitem{Broadhurst}
D.~Broadhurst.
\newblock Dickman polylogarithms and their constants.
\newblock {\em arXiv preprint arXiv:1004.0519}, 2010.

\bibitem{Franze}
C.~S. Franze.
\newblock A family of multiple integrals connected with relatives of the
  {D}ickman function.
\newblock {\em J. Number Theory}, 179:33--49, 2017.

\bibitem{Grupp}
F.~Grupp and H.-E. Richert.
\newblock The functions of the linear sieve.
\newblock {\em J. Number Theory}, 22(2):208--239, 1986.

\bibitem{Smith}
A.~Smith.
\newblock $2^{\infty}$-{S}elmer groups, $2^{\infty}$-class groups, and
  {G}oldfeld's conjecture.
\newblock {\em arXiv preprint arXiv:1702.02325}, 2017.

\bibitem{Sound}
K.~Soundararajan.
\newblock An asymptotic expansion related to the {D}ickman function.
\newblock {\em Ramanujan J.}, 29(1-3):25--30, 2012.

\end{thebibliography}
\end{document}